\documentclass[12pt]{amsart}

\usepackage{amssymb}
\usepackage{pb-diagram}
\usepackage{enumitem}
\usepackage[dvipsnames]{xcolor}
\usepackage{graphicx}
\usepackage{amsmath}
\numberwithin{equation}{section}
\usepackage[all]{xy}

\newtheorem{Theorem}{Theorem}[section]
\newtheorem{Proposition}[Theorem]{Proposition}
\newtheorem{Lemma}[Theorem]{Lemma}

\newtheorem{Definition}[Theorem]{Definition}
\newtheorem{Remark}[Theorem]{Remark}

\numberwithin{equation}{section}

\begin{document}

\baselineskip=16pt

\title{On the Hilbert scheme of the moduli space of torsion free sheaves on surfaces.}

\author{O. Mata-Guti\'errez}

\address{Departamento de Matem\'aticas\newline
Centro Universitario de Ciencias Exactas e Ingenier\'ias \newline
Universidad de Guadalajara\\
\newline  Avenida Revoluci\'on 1500\\ Guadalajara, Jalisco, M\'exico.}
\email{osbaldo.mata@academicos.udg.mx}

\author{L. Roa-Leguizam\'on}

\address{Universidad de Los Andes, Departamento de matem\'aticas \newline Carrera 1 \#18A-12, 111 711   \newline Bogota, Colombia.}
\email{leonardo.roa@cimat.mx}

\author{H. Torres-L\'opez}

\address{CONACyT - U. A. Matem\'aticas, U. Aut\'onoma de
Zacatecas
\newline  Calzada Solidaridad entronque Paseo a la
Bufa, \newline C.P. 98000, Zacatecas, Zac. M\'exico.}

\email{hugo@cimat.mx}

\thanks{The first author acknowledges
the financial support of Universidad de Guadalajara via PROSNI programme.
%Fondo Institucional de Fomento Regional
%para el Desarrollo Cient\'ifico, Tecnol\'ogico y de Innovaci\'on,
%FORDECYT 265667.
}

\subjclass[2010]{}


\date{\today}

\begin{abstract}
The aim of this paper is to  determine a bound of the dimension of an irreducible component
of the Hilbert scheme of the moduli space of torsion-free sheaves on surfaces.
Let $X$ be a non-singular irreducible complex surface and let $E$ be a vector bundle
of rank $n$ on $X$.  We use the $m$-elementary transformation of $E$ at a point
$x \in X$ to show that there exists an embedding from the Grassmannian variety
$\mathbb{G}(E_x,m)$ into the moduli space of torsion-free sheaves
$\mathfrak{M}_{X,H}(n;c_1,c_2+m)$ which induces an injective morphism from
$X \times M_{X,H}(n;c_1,c_2)$ to  $Hilb_{\, \mathfrak{M}_{X,H}(n;c_1,c_2+m)}$.
\end{abstract}

\maketitle

\section{introduction}

Let $X$ be a non-singular irreducible complex projective variety of dimension $d$.
Let $E$ be a vector bundle of rank $n$ and fixed Chern classes $c_i \in H^{2i}(X, \mathbb{Z})$ on $X$.
The $m$-elementary transformation $E'$ of $E$ at the point $x \in X$   is defined as the kernel of a
surjection $\alpha:E \longrightarrow \mathcal{O}_x^m$ which fits the exact sequence

\begin{equation}\label{Transformacion}
    0 \to E' \to E \to \mathcal{O}_x^m \to 0.
\end{equation}

It is not hard to check that the class of  such extensions are parameterized by $\mathbb{G}(E_x,m)$.
%moreover, since $E$ is locally free the $m$-elementary transformation of $E$ is not locally free
%whenever $\text{dim}(Y) \geq 2$.
This elementary transformations coincides with those defined by Maruyama, when $X$ is
a curve (see, \cite{Maruyama1}) but differs when $\dim\, X\geq 2$, because the point $x\in X$ is
not a divisor anymore.

Maruyama used his definition of elementary transformation to construct vector bundles
on non-singular projective varieties. Since then these elementary transformations have been
a powerfull tool in order to get topological and geometric properties of the moduli space
of sheaves, for instance:

When $X$ is a curve and $m=1$, the elementary transformation $E'$ of $E$ is a vector bundle.
Moreover, if $E$ is a general stable vector bundle  then $E'$ is stable and under this condition
Narasimhan and Ramanan  used elementary transformations to determine certain subvarieties
(called Hecke cycles) in the moduli space of vector bundles on curves, see \cite{Narasimhan-Ramanan, NR}.
These Hecke cycles are contained in a component of the Hilbert scheme
of the moduli space of vector bundles on curves (called Hecke component).
Hence, Narasimhan and Ramanan computed a bound for the dimension of the Hecke component
and proved that is non-singular in those points defined by Hecke cycles.
Moreover, when $X$ is a curve and $m\geq2,$ Brambila-Paz and Mata-Guti\'errez in \cite{Brambila-Mata}
generalized the construction of Hecke cycles using Grassmannians and
defined Hecke Grassmannians. They proved that the corresponding Hecke component
is non-singular and a bound for its dimension was given.

In case that $X$ is a surface and $m=1$, Coskun and Huizenga in \cite{Coskun-Huizenga} used elementary modifications to determine a component of the moduli space of  vector bundles
of rank two and compute a bound for its dimension.  Also, Costa and Mir\'o-Roig used priority sheaves and
elementary transformations in the sense of Maruyama in order to  establish maps
between certain moduli spaces over $\mathbb{P}^2$ with the same rank and different Chern classes
(see \cite{CM}).\\

The aim of this paper is to consider the case when $X$ is a surface and $m\geq 1$, we
use $m$-elementary transformations to determine Hecke cycles in the moduli space of
stable torsion free sheaves and determine geometrical aspects of a component of
its Hilbert scheme. Specifically, we  prove the following result (see Theorem \ref{principal}):

\begin{Theorem}
The Hilbert scheme $\text{Hilb}_{\mathfrak{M}_{X,H}(n;c_1,c_2+m)}$ of the moduli space of stable
torsion-free sheaves has an irreducible component of dimension at least $2+\dim\, \mathfrak{M}_{X,H}(n;c_1,c_2+m)$.
\end{Theorem}

The proof of this Theorem follows some ideas and techniques of \cite{Brambila-Mata} and \cite{Narasimhan-Ramanan}. For a fixed vector bundle $E$ and a point $x\in X$ we determine a
closed embedding $\phi_z: \mathbb{G}(E_x,m) \mapsto \mathcal{M}_{X,H}(n;c_1,c_2)$  (see Proposition \ref{encajeG}).
We use the closed embedding $\phi_z$ to define the injective morphism \[\begin{aligned}
    \psi: X \times M_{X,H}(n;c_1,c_2) & \longrightarrow \text{Hilb}_{\mathfrak{M}_{X,H}(n;c_1,c_2+m)} \\
    z & \mapsto \phi_z(\mathbb{G}(E_x,m)).
\end{aligned}\]

Additionally, we establish the following morphism \[\Phi:\mathbb{G}(\mathcal{U},m) \rightarrow \mathfrak{M}_{X,H}(n;c_1,c_2+m)\]
where $\mathcal{U}$ denotes the universal family parameterized by $M_{X,H}(n;c_1,c_2)$. This morphism allows us to
determine an irreducible projective variety of $\mathfrak{M}_{X,H}(n;c_1,c_2+m)-M_{X,H}(n;c_1,c_2+m)$ and we get the following result (see Theorem \ref{second}):

\begin{Theorem}
Let $m,n$ natural integers with $1\leq m < n $. Then $\mathfrak{M}_{X,H}(n;c_1,c_2)-M_{X,H}(n;c_1,c_2+m)$ contains an irreducible projective variety
$Y$ of dimension $3+\dim\, M_{X,H}(n;c_1,c_2)$ such that the general element
$F \in Y$ fits into exact sequence
\begin{eqnarray*}
0 \to F \to E \to \mathcal{O}_{X,x}\otimes W \to 0,
\end{eqnarray*}
    where $E\in M_{X,H}(n;c_1,c_2)$, $W\in\mathbb{G}(E_x, m)$ and $x\in X$.
    In particular, if $n=2 $ then $\Phi$ is injective and $Y$  is a divisor.
\end{Theorem}

As an application of the previous result we compute the Hilbert polynomial of the Hilbert scheme $\text{Hilb}^P_{\mathfrak{M}_{X,H}(n;c_1,c_2)}$ which contains the cycle  $\phi_z(\mathbb{G}(E_x,m))$ when $X$ is the projective plane. In particular, we prove the following (see Theorem \ref{princlemma});

\begin{Theorem}
Assume that $c_1=-1$ (resp. $c_1=0$) and that
$c_2 \geq 2$ (resp. $c_2 \geq 3$ is odd). Let  $L= a\epsilon + b\delta$, (resp. $a\varphi+b\psi$ ) be an ample line bundle in $Pic(\mathfrak{M}_{\mathbb{P}^2}(2;c_1,c_2))$. Then,  $\mathcal{H}\mathcal{G}$  is the component of the Hilbert scheme $\text{Hilb}^P_{\mathfrak{M}_{\mathbb{P}^2}(2;c_1,c_2)}$  where $P$ is the Hilbert polynomial defined as;
 \begin{eqnarray*}
P(m) = \chi(\mathbb{P}(E_x), \phi^{*}_z(a\epsilon+b\delta) )&=& \chi(\mathbb{P}(E_x),  \mathcal{O}_{\mathbb{P}(E_x)}(mb)). \\
(resp. \,\,\, \, \,\, P(m) = \chi(\mathbb{P}(E_x), \phi^{*}_z(a\varphi+b\psi) )&=& \chi(\mathbb{P}(E_x),  \mathcal{O}_{\mathbb{P}(E_x)}(m(c_2-1)b))).
\end{eqnarray*}
\end{Theorem}

The paper is organized as follows:  Section 2 contains a brief summary of the main
results of Grassmannians of vector bundles, moduli space of torsion-free sheaves and $m$-elementary transformations. In Section 3, we give some technical results which all\-ow us
to prove our main results: Theorem 1.1 and Theorem 1.2. In Section 4. an application of
the previous results is indicated for the Hilbert scheme of moduli space of rank $2$ sheaves
on the projective plane.

%Hence if we denote by $X$ such surface and $x\in X$, then we consider an elementary transformation of $H$-stable vector bundle $E$ in $x$, we have$$0\rightarrow F\rightarrow E \rightarrow \mathcal{O}^{m}_{x} \rightarrow 0.$$  Later, in \cite{} the authors used the ideas of Narasimhan and Ramanan to introduce new varieties in the moduli space of vector bundles         called Hecke grassmannians and concluded  thatthe Hilbert scheme of the moduli space of vector bundleshas a non singular variety of dimension, see \cite{}.

\section{preliminaries}

 Let $X$ be a non-singular  irreducible, complex, projective algebraic surface. This section contains a brief summary about stable torsion free sheaves on  surfaces, and we recall some
 basic facts on Grasmannians of vector bundles  and $m$-elementary transformations see  \cite{Fantechi, Friedman, Huybrechts-Lehn} for more details.

 \subsection{Grassmannian}
We will collect here the principal properties of Grassmannians of vector bundles necessary for our purpose. For a fuller treatment we refer the reader to \cite{Eisenbud, Tyurin}.

Let $E$ be a vector bundle of rank $n$ on $X$. Let $p_E: \mathbb{G}(E,m)\rightarrow X$ be the Grassmannian bundle of rank $m$ quotients of $E$ whose fiber at  $x\in X$ is the Grassmannian
$\mathbb{G}(E_{x},m)$ of $m$-dimensional quotients  of $E_x$, that is
\[
\mathbb{G}(E,m)=\{(x,W)\ |\ x\in X,\ E_x\rightarrow W\rightarrow 0 \}.\]

Let
\[
0\rightarrow S_E\rightarrow p^{*} E \rightarrow Q_E\rightarrow 0
\]
be the tautological exact sequence over $\mathbb{G}(E,m)$  where $S_E$ and $Q_E$ denote the universal
subbundle of rank $n-m$ and  universal quotient of rank $m$, respectively.   The tangent bundle of $\mathbb{G}(E,m)$  is the vector bundle $T\mathbb{G}(E,m)=Hom(S_E,Q_E)$ and
hence $T_x\mathbb{G}(E,m)=Hom(S_{E_x},Q_{E_x})$. Moreover, we have the following exact sequence:
\[
0\rightarrow T_{p_E}\rightarrow T\mathbb{G}(E,m)\rightarrow p^{*}_E TX\rightarrow 0
\]
where $T_{p_E}$ is the relative tangent bundle to the fibers and $T_{p_E}=S^{*}_E\otimes Q_E$.

 \subsection{Torsion-Free sheaves}

    Let $H$ be  an ample divisor  on $X$. For a  torsion-free sheaf $\mathcal{E}$ on $X$
    with Chern classes $c_i \in H^{2i}(X,\mathbb{Z})$, $i=1,2$ one sets
    \[
    \mu_H(\mathcal{E}) :=\frac{\text{deg}_H(\mathcal{E})}{\text{rk}(\mathcal{E})}, \,\,\,\,\,\, P_m(\mathcal{E}):= \frac{\chi(\mathcal{E} \otimes H^m)}{\text{rk}\, (\mathcal{E})},
    \]
    where $\text{deg}_H(\mathcal{E})$ is the degree of $\mathcal{E}$ defined by $c_1(\mathcal{E}).H$
    and $\chi(\mathcal{E} \otimes H^m)$ denotes the Hilbert polynomial defined by
    $\sum (-1)^ih^i(X, \mathcal{E} \otimes H^m)$.

    \begin{Definition}
    \emph{Let $H$ be an ample divisor on $X$. A torsion-free sheaf $\mathcal{E}$ on $X$ is $H$-stable
    (resp. stable) if for all non-zero subsheaf $\mathcal{F}\subset \mathcal{E}$
    \begin{eqnarray*}
    \mu_H(\mathcal{F})< \mu_H(\mathcal{E}) \ \  ( resp. \,\, P_m(\mathcal{F}) < P_m(\mathcal{E})).
    \end{eqnarray*}}
    \end{Definition}

We want to emphasize that both notions of stability depend on the ample divisor  we
fix on the underlying surface $X$ and it is easily  seen that $H$-stability implies stability.\footnote{The $H$-stability is frequently called Mumford-Takemoto stability and the stability is called Gieseker-Maruyama stability.}

Recall that any  $H$-stable (resp. stable) torsion-free sheaf is simple, i.e. if $\mathcal{E}$ is
$H$-stable (resp. stable), then $\dim Hom(\mathcal{E},\mathcal{E})=1$.  We will denote by $M_{X,H}(n; c_1,c_2)$ the
moduli space of $H$-stable vector bundles on $X$ of rank $n$ and fixed Chern classes $c_1, c_2$
and by  $\mathfrak{M}_{X,H}(n;c_1,c_:2)$ the moduli space of stable torsion free sheaves on $X$. % with fixed Hilbert polynomial $P: = P_m(\mathcal{E})$.
Since locally free is an open property and $H$-stability implies stability, it follows that $M_{X,H}(n; c_1,c_2)$ is an open subset of  $\mathfrak{M}_{X,H}(n;c_1,c_2)$. In general an universal family on $X \times M_{X,H}(n;c_1,c_2)$ (resp. on $X \times \mathfrak{M}_{X,H}(n;c_1,c_2)$) does not exist, the existence of such universal family is guaranteed by the following criterion;

\begin{Lemma}\cite[Corollary 4.6.7]{Huybrechts-Lehn} \label{Universal}
Let $X$ be a non-singular surface and let $H$ be an ample
divisor on $X$.  Let $n,c_1, c_2$ fixed values for the rank and Chern classes.
If $\text{gcd}(n, c_1.H, \frac{1}{2}c_1.(c_1-K_X)-c_2)=1$, then there is an universal family
on $X \times M_{X,H}(n;c_1,c_2)$ (resp.  $X \times \mathfrak{M}_{X,H}(n;c_1,c_2)$).
\end{Lemma}

\subsection{$m$-elementary transformations}

\begin{Definition} \emph{
Let $E$ be a locally free sheaf on $X$ of rank $n$ and Chern classes $c_1, c_2$
and let
\begin{equation} \label{ET}
    0\rightarrow E'\rightarrow E \rightarrow \mathcal{O}^{m}_{x}\rightarrow 0
\end{equation}
be an exact sequence of sheaves, where $\mathcal{O}^{m}_{x}=\oplus_{i=1}^{m} \mathcal{O}_{x}$
is the sum of skyscraper sheaf with support on $x\in X.$
The coherent  sheaf $E'$ is called the $m$-elementary transformation of $E$ at
$x \in X$.}
\end{Definition}

Notice that even though $E$ is locally free, its elementary transformation $E'$
is a torsion free sheaf not locally free. Moreover if $E$ is $H$-stable then
$E'$ is also $H$-stable. However if $E$ is  stable then $E'$ is not
necessarily   stable (see for instance \cite[Remark 1]{Coskun-Huizenga3}).

The $m$-elementary transformations have been used for several authors
to construct many vector bundles on a higher
dimensional projective variety and to determine topological and geometric properties of the moduli space of sheaves. For instance,
Maruyama did a general study of elementary transformations of sheaves in his master’s and doctoral theses \cite{Maruyama1}. In \cite{Narasimhan-Ramanan} Narasimhan and Ramanan used elementary
transformations of vector bundles on curves to introduce certain subvarieties
in the moduli space of vector bundles which they called Hecke cycles.
Brambila-Paz and the first author also used $m$-elementary transformations to describe
a non-singular open set of the Hilbert scheme of the moduli space of vector bundles
on a curve \cite{Brambila-Mata}.  Coskun and Huizenga have  been used elementary transformations
to study priority sheaves since
that they are well-behaved under elementary modifications \cite{Coskun-Huizenga, Coskun-Huizenga1, Coskun-Huizenga2}.

We now collect some other basic properties related with $m$-elementary transformations in the following result.

\begin{Proposition} \label{properties}
Let $H$ be an ample divisor on  $X$.  Let $E$ be a vector bundle on $X$ of rank $n$
and Chern classes $c_1,c_2$, and let $E'$ be a
$m$-elementary transformation of $E$ at  $x \in X,$ i.e. we have
\begin{equation} \label{trans.elemental}
    0 \rightarrow E' \rightarrow E \rightarrow \mathcal{O}_x^m \rightarrow 0.
\end{equation}
Then,
\begin{itemize}
    \item [(i)]  $rk(E')=n$, $c_1(E')=c_1$, $c_2(E')= c_2+m$ and $\chi(E')=\chi(E)-m.$
    \item [(ii)] $E'$ is a torsion-free sheaf not locally free.
    \item [(iii)] If $E$ is $H$-stable, then $E'$ is $H$-stable. Hence, $E'$ is stable.
\end{itemize}
\end{Proposition}

\begin{proof}
\begin{itemize}
\item [(i)] The proof follows directly from the exact sequence and Riemann-Roch Theo\-rem.
    \item [(ii)] Clearly $E'$ is torsion free since $E$ is a vector bundle.
    Now, suppose that $E'$ is locally free, by  \cite[Chapter 4, Lemma 3]{Friedman}, it follows that $E =E'$ which is impossible because $c_2(E') =c_2+m$.
    Therefore $E'$ is a torsion free sheaf  not locally free.
    \item [(iii)] Let $F$ be subsheaf of $E'$ and assume that $E$ is $H$-stable.
    It is clear that $F$ is a subsheaf of $E$ and  by item $(i)$, it follows that
    \[\mu_H(F) < \mu_H(E) = \mu_H(E').\]
    Hence $E'$ is $H$-stable and therefore stable.
\end{itemize}
\end{proof}

\begin{Remark} \emph{
The class of extensions (\ref{trans.elemental}) are parameterized by $\mathbb{G}(E_x,m)$.  Furthermore,  any $W \in \mathbb{G}(E_x,m)$
defines a surjective linear transformation
$\tilde{\alpha}_{W}:E_{x}\rightarrow W\rightarrow 0$ which determines
a surjective morphism of sheaves $\alpha_{W}:E\rightarrow \mathcal{O}^{m}_{x}$.
If $E^W$ denotes $\text{ker}(\alpha_W)$ then we have the exact sequence:
\begin{eqnarray}\label{hext}
0\rightarrow E^{W}\rightarrow E \rightarrow \mathcal{O}^{m}_{x}\rightarrow 0.
\end{eqnarray}}
\end{Remark}

The following result will be used in the next sections: %is obtained using Koszul complexes and can be consulted in \cite{}.}

\begin{Lemma} \label{Ext}
  Let $E$ be a vector bundle on $X$ and let  $\mathcal{O}_{x}$ be the  skyscraper sheaf
  with support on $x\in X$. Then, for any integer $m\geq 1$ we have
\[
\text{Ext}^{i}(\mathcal{O}^{m}_{x},E)=0, \ \ \ \ i\neq 2.
\]
\end{Lemma}

For a deeper discussion of $m$-elementary transformations we refer to reader to  \cite{Brambila-Mata, Coskun-Huizenga}.

\section{On the moduli space of torsion free sheaves}

The aim of this section is to define an embedding from $\mathbb{G}(E_x,m)$
into the moduli space $\mathfrak{M}_{X,H}(n;c_1,c_2+m)$ of torsion free sheaves.  Generalizing some techniques of  \cite{Brambila-Mata} and \cite{Narasimhan-Ramanan}  we  establish a closed embedding $\phi_z: \mathbb{G}(E_x, m) \to \mathfrak{M}_{X,H}(n;c_1,c_2+m)$ and
an injective algebraic morphism
$\Psi:X \times M_{X,H}(n;c_1, c_2) \to  \text{Hilb}_{\mathfrak{M}_{X,H}(n;c_1,c_2+m)},$
where $z =(x,E) \in X \times M_{X, H}(n;c_1,c_2)$ and $Hilb_{\mathfrak{M}_{X,H}(n;c_1,c_2+m)}$
denotes the Hilbert scheme of the moduli space $\mathfrak{M}_{X,H}(n;c_1,c_2)$.
Moreover, we construct an irreducible variety properly contained  in $\mathfrak{M}_{X,H}(n;c_1,c_2+m)-M_{X,H}(n;c_1,c_2+m)$.

The following Lemma deals with $m$-elementary transformations, specifically we compute the dimension of the morphisms of  a $m$-elementary transformation $E'$ of $E$. The important point to note here is that $E$ is a vector bundle.  Here and sequentially, $E$ denotes a vector bundle on $X$.

\begin{Lemma}\label{INY}
Let $H$ be an ample divisor on $X$.
Let $E'$ be a torsion-free sheaf of rank $n$
and let $E$ be an $H$-stable vector bundle of rank $n$. If $c_1(E')= c_1(E)$, then
$\dim \, Hom (E',E) \leq 1.$
\end{Lemma}

\begin{proof}
Let $f:E'\rightarrow E$ be a not zero homomorphism.
By \cite[Proposition 7, Chapter 4]{Friedman} the morphism $f$ is injective and hence
we have the sequence
\[0 \rightarrow E' \rightarrow E \rightarrow E/E' \rightarrow 0.\]
By \cite[Proposition 6.4.]{Hartshorne1}, we have the following  long exact sequence
\[\begin{aligned}
0 \rightarrow & \text{Hom} (E/E^{'}, E) \rightarrow \text{Hom}(E,E) \rightarrow \text{Hom}(E^{'},E) \rightarrow \\ & \text{Ext}^1(E/E^{'}, E) \rightarrow \text{Ext}^1(E,E) \rightarrow \text{Ext}^1(E^{'},E) \rightarrow \cdots \end{aligned}\]

Note that $E/E'$ has support in a finite number of points
because $c_1(E)=c_1(E^{'})$, hence $\text{Hom}(E/E',E)=0.$
On the other hand Lemma \ref{Ext}, implies that $\text{Ext}^{1}(E/E',E)=0$. Since $E$ is a $H$-stable vector bundle,
it follows that
\[ \dim\, \text{Hom}(E,E) = \dim\, \text{Hom}(E^{'},E) = 1\]
as we desired.
\end{proof}
%We claim that
%\[Hom (E/E^{'}, E) = \text{Ext}^1(E/E^{'}, E)= 0.\]
%First note that, since $E$ is locally free, then there exist  the following isomorphisms
%\[Hom(E/E^{'}, E) \cong \text{Ext}^0(E/E^{'}, E) \cong Ext^2(   \ \ E, E/E^{'} \otimes K_X)^* \cong H^2(X, E^* \otimes E/E^{'} \otimes K_X)\]
%and
%\[Ext^1(E/E^{'}, E) \cong Ext^1( E, E/E^{'} \otimes K_X)^* \cong H^1(X, E^* \otimes E/E^{'} \otimes K_X).\]
%\textcolor{red}{Se usó el gootsche prop1.1 y el Hartshorne 6.3 y 6.7 pág 234, al parecer solo se aplica para torsion free. Verificar esto.
%Ver Hartshorne 243 thm 7.6}

%The claim follows from $E/E^{'}$ is a sheaf supported in points because $c_1(E)=c_1(E^{'}).$ Since $E$ is $H$-stable, it follows that $E$ is simple. Therefore,
%\[ \text{dim} \, Hom(E,E) = \text{dim} \, Hom(E^{'},E) = 1\]
%as we desired.

Set $z:=(x,E)\in X \times M_{X,H}(n;c_1,c_2)$ and let $m$ be a fixed
natural number with $m<n$.  Let $\pi_E: \mathbb{G}(E,m) \rightarrow X$ be the
Grassmannian bundle associated to $E$ and for any  $x \in X$ denote by  $\mathbb{G}(E_x,m)$
the Grassmannian of $m$-quotients of $E_x$. On $\mathbb{G}(E,m)$ we
have the tautological exact sequence
\begin{eqnarray}\label{tautseq}
0 \rightarrow S_E \rightarrow \pi_E^*E \rightarrow Q_E \rightarrow 0,
\end{eqnarray}
where $S_E$ is the universal subbundle and $Q_E$ is the universal quotient bundle.
Note that for any $x \in X$, if we restrict  (\ref{tautseq}) to  $\mathbb{G}(E_x,m)$ then we obtain
\begin{equation} \label{rest}
    0 \rightarrow S_{E_{x}} \rightarrow \mathcal{O}_{\mathbb{G}} \times E_x \rightarrow Q_{E_x} \rightarrow 0.
\end{equation}
Let us denote by $\mathbb{G}(z) := \mathbb{G}(E_x,m)$. Consider on
$X \times \mathbb{G}(z)$ the surjective morphism $\alpha: p_1^*E \longrightarrow p_1^*\mathcal{O}_x \otimes p_2^*Q_{E_x}$ associated to the canonical surjective morphism $\alpha_x: \mathcal{O}_{\mathbb{G}} \times E_{x} \rightarrow Q_{E_x}$ in $(\ref{rest})$ under  the isomorphism;
\begin{eqnarray*}
H^0(X \times \mathbb{G}(z), p_1^{*}E^* \otimes p_1^{*}\mathcal{O}_x \otimes p_{2}^*Q_{E_x})
&\cong & H^0( \mathbb{G}(z), p_{2_*} (p_1^{*}E^* \otimes p_1^*\mathcal{O}_x ) \otimes Q_{E_x})\\
&\cong & H^0( \mathbb{G}(z), p_{2_*}p_1^*(E_{{x}}^*) \otimes Q_{E_x}) \\
&\cong & H^0( \mathbb{G}(z), (\mathcal{O}_{\mathbb{G}} \times E_x^*) \otimes Q_{E_x}) \\
&\cong & H^0( \mathbb{G}(z), \mathcal{H}om(\mathcal{O}_{\mathbb{G}} \times E_x, Q_{E_x})),
\end{eqnarray*}
where the second isomorphism is given by projection formula (see, \cite{mumford}, page 76).
Here, taking the kernel of the surjective morphism
$\alpha: p_1^*E \longrightarrow p_1^*\mathcal{O}_x \otimes p_2^*Q_{E_x},$
we get the exact sequence
\begin{eqnarray}\label{famE}
0 \longrightarrow \mathcal{F}_z
\longrightarrow p_1^*E \longrightarrow p_1^*\mathcal{O}_x \otimes p_2^*Q_{E_x} \longrightarrow 0
\end{eqnarray}
on $X \times \mathbb{G}(z)$.

   \begin{Lemma}\label{tor1} Let $z=(x,E)\in X\times M_{X,H}(n; c_1,c_2)$ and $W\in\mathbb{G}(z)$, then
   \begin{eqnarray*}
    \mathcal{T}or^1(\mathcal{O}_{\{x\} \times \mathbb{G}}, \mathcal{O}_{X \times \{W\}})=0.
   \end{eqnarray*}
   \end{Lemma}
\begin{proof}
Restricting the exact sequence
\begin{eqnarray*}
0\to I_{\{x\} \times \mathbb{G}}\to \mathcal{O}_{X \times \mathbb{G}} \to \mathcal{O}_{\{ x\}\times \mathbb{G}}\to 0
\end{eqnarray*}
to $X \times \{W\}$, we get
\begin{eqnarray*}
            0 \to \mathcal{T}or^1(\mathcal{O}_{\{x\} \times \mathbb{G}}, \mathcal{O}_{X \times \{W\}})\to I_{\{x\} \times \mathbb{G}}\vert_{X\times \{W\}}\to \mathcal{O}_{X} \to \mathcal{O}_{ x}\to 0
\end{eqnarray*}
As is well known $p_1^{*}I_x\cong I_{\{x\}\times \mathbb{G}}$ and
 $ I_{\{x\} \times \mathbb{G}}\vert_{X\times \{W\}}\cong I_x$.
% (see for example  \ref{Vakil}, Vakil).
Then it follows that

$$ \mathcal{T}or^1(\mathcal{O}_{\{x\} \times \mathbb{G}}, \mathcal{O}_{X \times \{W\}})=0.
$$
\end{proof}

With the above notation  and as consequence of Lemma \ref{tor1}, we have the following result.

\begin{Proposition} \label{Family} If $E$ is $H$-stable, then
 $\mathcal{F}_z$ is a family of stable
 torsion free sheaves parameterized by $\mathbb{G}(z)$.
\end{Proposition}
\begin{proof}
Let $W\in \mathbb{G}(z)$. Restricting the exact sequence
$(\ref{famE})$  to $X\times \{W\}$,
we get  the exact sequence
\begin{eqnarray}\label{famEinW}
0 \longrightarrow E^{W} \longrightarrow E \longrightarrow \mathcal{O}_x \otimes W \longrightarrow 0
\end{eqnarray}
over $X$.  Hence, $E^{W}$ is a torsion-free sheaf of rank $n$ called
the $m$-elementary transformation of $E$ in $x$ defined by $W$.
Since $c_1( \mathcal{O}_x \otimes W)=0$ and $E$ is $H$-stable, it follows that $E^W$
is $H$-stable and therefore stable with $c_1(E^W)=c_1(E)$ (see Proposition \ref{properties}).
Moreover, by Whitney sum and $c_2( \mathcal{O}_x \otimes W)=-\dim\, (W)=-m$ we get $c_2(E^{W})=c_2(E)+m$ which completes the proof.
\end{proof}

%Let $H$ be an ample divisor on $X$.
The classification map of $\mathcal{F}_z$ is given by
\[\begin{aligned}\label{classmap}
\phi_z: \mathbb{G}(z)&\rightarrow  \mathfrak{M}_{X,H}(n;c_1,c_2+m)\\
W&\mapsto  E^{W},
\end{aligned}\]
where $E^W$ was defined in the above Proposition %and $P:= P_l(E^W)= \frac{\chi(E^W \otimes H^l)}{ \text{rk} E^W}$.
The following proposition shows  that the morphism $\phi_z$ is a closed embedding.  For the proof of the proposition we follow the techniques and ideas of \cite[Lemma 5.10]{Narasimhan-Ramanan}, and \cite[Proposition 3.1]{Brambila-Mata} who proved a similar result for vector bundles on curves.

\begin{Proposition}\label{encajeG}
For any  point $z=(x,E)\in X\times M_{X,H}(n;c_1,c_2)$, the morphism $\phi_z: \mathbb{G}(z)\rightarrow \mathfrak{M}_{X,H}(n;c_1,c_2+m)
$ is a closed embedding.
\end{Proposition}

\begin{proof}
We first prove that the morphism $\phi_z$ is injective.  Assume that there exist $W_1,W_2\in \mathbb{G}(z) $ such that  $\psi: E^{W_1}\rightarrow  E^{W_{2}}$
is an isomorphism, we claim that $W_1=W_2$.  Recall that for any $i=1,2$ we have the following exact sequence
\begin{eqnarray*}
    \xymatrix{0\ar[r]&E^{W_i}\ar[r]^{f_i}&E\ar[r]^{\alpha_{i}\ \ \ }&  \mathcal{O}_{x} \otimes W_i \ar[r]&0}.
    \end{eqnarray*}
By Lemma  \ref{INY} we have $\dim\,\text{Hom}(E^{W_1},E)=1$,  it follows that there exist
$\lambda \in \mathbb{C}^{*}$ such that  $ \lambda f_1=f_2\circ \psi $.
Hence $\text{Im}\,f_{1,x}=\text{Im}\,f_{2,x}$ which implies $W_{1}=W_{2}$.
Therefore $\phi_{z}$ is injective.

We now proceed to show the injectivity of the differential map  $d\phi_z: T_W\mathbb{G}(z) \to \mathfrak{M}_{X,H}(n;c_1,c_2+m)$. By \cite[Lemma 5.10]{Narasimhan-Ramanan}, its infinitesimal deformation map
in $W\in \mathbb{G}(z)$ is, up to the sign, the composition of the natural map
$T_W\mathbb{G}(z) \rightarrow \text{Hom} (E^{W}, \mathcal{O}_{x}\otimes W)$ with the boundary map
$\text{Hom} (E^{W}, \mathcal{O}_{x}\otimes W)\rightarrow \text{Ext}^{1}(X, E^{W},E^{W})$
given by the long exact sequence
\begin{eqnarray*}
 0\rightarrow \text{Hom}(E^{W},E^{W})\rightarrow \text{Hom}(E^{W},E)\rightarrow \text{Hom} (E^{W}, \mathcal{O}_{x}\otimes W) \rightarrow \text{Ext}^{1}(E^{W},E^{W})\rightarrow\cdots
\end{eqnarray*}
obtained from  (\ref{famEinW}). Notice that $\text{Hom}(E^{W},E^{W})\cong \mathbb{C}$ because $E^{W}$ is an $H$-stable
free torsion sheaf. Moreover, $\text{Hom}(E^{W},E)\cong \mathbb{C}$ by Lemma \ref{INY}.
Therefore the coboundary morphism
$$
\delta: \text{Hom} (E^{W}, \mathcal{O}_{x}\otimes W)\rightarrow \text{Ext}^{1}(E^{W},E^{W})
$$
is injective.
\end{proof}

As in \cite{Brambila-Mata, Narasimhan-Ramanan}, a consequence of the above result is that we determine a collection of
closed subschemes in  $\mathfrak{M}_{X,H}(n;C_1,c_2+m)$ and a collection of points in its Hilbert scheme (see, \cite [Definition 5.12]{Narasimhan-Ramanan}.
From a stable vector bundle $E$ on X, we constructed the family $\mathcal{F}_z$ of stable torsion free sheaves. Analogously, if we start with a family $\mathcal{E}$ of stable vector bundles on $X$ parameterized by $T$, then we can construct a family of of stable torsion free sheaves $\mathcal{F}$. In the next paragraphs we describe the construction when $\mathcal{E}$ is the universal family of stable vector bundles parameterized by $M_{X,H}(n;c_1, c_2)$.

Let $H$ be an ample divisor on $X$.
As is well known if  $\text{gcd}(n, c_1.H, \frac{1}{2}c_1.(c_1-K_X)-c_2)=1,$
then there exists a universal family $\mathcal{U}$ of vector bundles parameterized by $M_{X,H}(n;c_1,c_2)$
(see Lemma \ref{Universal}).
Under this conditions we will determine
a family $\mathcal{F}$ of stable torsion-free sheaves parameterized by $\mathbb{G}(\mathcal{U},m)$ which extends to $\mathcal{F}_z$ (see Proposition (\ref{Family})).

Let $\mathcal{U}$ be the universal family of vector bundles parameterized by
$M_{X,H}(n;c_1,c_2)$, hence $p:\mathcal{U}\rightarrow X\times M_{X,H}(n;c_1, c_2)$
is a vector bundle.  We denote  by $\pi_{\mathcal{U}}:\mathbb{G}(\mathcal{U},m) \rightarrow X\times M_{X,H}(n; c_1, c_2) $  the Grassmannian bundle of quotients associated to
$\mathcal{U}$. An element of $\mathbb{G}(\mathcal{U}, m)$ is a pair $((x,E),W),$
where $(x,E) \in X \times M_{X,H}(n;c_1,c_2)$ and $W\in \mathbb{G}(E_x,m)$.
The tautological exact sequence  over $\mathbb{G}(\mathcal{U}, m)$ is
\begin{eqnarray}\label{tautsequniv}
0 \rightarrow S_\mathcal{U} \rightarrow \pi_{\mathcal{U}}^*\,\mathcal{U}\stackrel{\alpha} \rightarrow Q_{\mathcal{U}} \rightarrow 0,
\end{eqnarray}
 where $Q_{\mathcal{U}}$ denotes the
universal quotient bundle of rank $m$ over $\mathbb{G}(\mathcal{U},m)$.
We now consider the graph of the following composition
$$
\xymatrix{
\mathbb{G}(\mathcal{U},m)\ar[r]^-{\pi_{\mathcal{U}} }& X\times M_{X,H}(n; c_1, c_2) \ar[r]^-{p_1} &X,
}
$$
$\Gamma:=\Gamma_{p_{1}\circ \pi_{U}}$
as a subvariety of $X\times \mathbb{G}(\mathcal{U},m)$. Then we have the following result;
\begin{Lemma}Let $g\in \mathbb{G}(\mathcal{U},m)$. Then
\begin{itemize}
    \item[(a)] $\mathcal{T}or^1(I_{X\times \{g\}}, \mathcal{O}_{\Gamma})=0.$
    \item[(b)]
There exists  a canonical surjective morphism of sheaves
\begin{equation}\label{surjective}
(id \times  p_2 \circ \pi_{\mathcal{U}}) ^{*}\,\mathcal{U}\rightarrow \mathcal{O}_{\Gamma}\otimes p_{\mathbb{G}(\mathcal{U})}^{*}Q_{\mathcal{U}}\rightarrow 0,
\end{equation}

over $X\times \mathbb{G}(\mathcal{U},m),$ determined by $\alpha$, where
$p_{\mathbb{G}(\mathcal{U})}:X\times \mathbb{G}(\mathcal{U},m) \rightarrow \mathbb{G}(\mathcal{U},m)$ and
$p_2:X\times M_{X,H}(n;c_1,c_2) \rightarrow M_{X,H}(n;c_1, c_2)$
are the respective second projections.
\end{itemize}
\end{Lemma}

\begin{proof}
Taking $\beta:=p_{\mathbb{G}(\mathcal{U})}|_\Gamma$ as the restriction of the projection,  we have the following commutative diagram
  \[\begin{diagram}
  \node{\Gamma} \arrow[2]{e,l}{ i} \arrow{se,l}{\beta} \node{} \node{X \times \mathbb{G}(\mathcal{U})}\arrow{sw,l}{p_{\mathbb{G}(U)}} \\
  \node{} \node{\mathbb{G}(\mathcal{U})}
  \end{diagram},\]
where $i:\Gamma \to  X\times \mathbb{G}(\mathcal{U})$ is the inclusion map, hence
$
I_{X \times {g}}\vert_{\Gamma}= i^{*}p_{\mathbb{G}(\mathcal{U})}^{*}(I_{g})=\beta^{*}(I_g).
$

From the exact sequence
\begin{eqnarray*}
0 \rightarrow   I_g \rightarrow   \mathcal{O}_{\mathbb{G}(\mathcal{U})} \rightarrow \mathcal{O}_{g}\rightarrow 0,
\end{eqnarray*}
we get
\begin{eqnarray*}
0 \rightarrow  \beta^{*}(I_g) \rightarrow   \beta^{*}(\mathcal{O}_{\mathbb{G}(\mathcal{U})}) \rightarrow \beta^{*}(\mathcal{O}_{g})\rightarrow 0,
\end{eqnarray*}
Therefore $\mathcal{T}or^1(I_{X\times \{g\}}, \mathcal{O}_{\Gamma})=0$ and this prove $(a)$.

Now, to prove $(b)$  consider the surjective map $\alpha: \pi_\mathcal{U}^{*}\,\mathcal{U}\rightarrow Q_{\mathcal{U}}$ given in (\ref{tautsequniv}) and notice that   $\beta^{*}\alpha:\beta^{*}\pi_\mathcal{U}^{*}\,\mathcal{U}\rightarrow \beta^{*}Q_{\mathcal{U}}$  is also surjective. Since
$\beta^{*} \pi_U^{*}(\mathcal{U})\cong (id \times   p_{2} \circ  \pi_{\mathcal{U}})^{*}\,(\mathcal{U})\vert_{\Gamma}$
and
$\beta^{*}Q_{\mathcal{U}}\cong p_{\mathbb{G}((\mathcal{U})}^{*}(Q_{\mathcal{U}})\vert_{\Gamma}$
 we get a surjective morphism
\begin{equation}\label{sur}
    (id \times   p_{2} \circ  \pi_{\mathcal{U}})^{*}\,(\mathcal{U})\vert_{\Gamma} \rightarrow \mathcal{O}_{\Gamma}\otimes p_{\mathbb{G}(\mathcal{U})}^{*}Q_{\mathcal{U}}.
\end{equation}
Hence from the exact sequence
\begin{eqnarray*}
0\to (id \times   p_{2} \circ  \pi_{\mathcal{U}})^{*}\,\mathcal{U}\ \otimes I_{\Gamma} \to (id \times   p_{2} \circ  \pi_{\mathcal{U}})^{*}\,\mathcal{U}\ \to (id \times   p_{2} \circ  \pi_{\mathcal{U}})^{*}\,\mathcal{U}\ \vert_{\Gamma}\to 0
\end{eqnarray*}
and the morphism (\ref{sur}) we get the surjective map
$(id \times   p_{2} \circ  \pi_{\mathcal{U}})^{*}\,\mathcal{U}\rightarrow \mathcal{O}_{\Gamma}\otimes p_{\mathbb{G}(\mathcal{U})}^{*}Q_{\mathcal{U}}$
which completes the proof.
\end{proof}

According to the above Lemma, let us denote by $\mathcal{F}$ the kernel of the surjective morphism  (\ref{surjective}). Hence we get the exact sequence
\begin{eqnarray}\label{sucuniversal}
0\rightarrow \mathcal{F}\rightarrow (id \times   p_{2} \circ  \pi_{\mathcal{U}})^{*}\,\mathcal{U}\rightarrow \mathcal{O}_{\Gamma}\otimes p_{\mathbb{G}(\mathcal{U})}^*Q_{\mathcal{U}} \rightarrow 0.
\end{eqnarray}
    Note that $(id \times   p_{2} \circ  \pi_{\mathcal{U}})^{*}\,(\mathcal{U}) \vert_{X\times ((x,E),W)}=E$ and $\mathcal{O}_{\Gamma}\otimes p_{\mathbb{G}(\mathcal{U})}^*Q_{\mathcal{U}}\vert_{X \times ((x,E),W)}=\mathcal{O}_x \otimes W$. Since $ p_{\mathbb{G}(\mathcal{U})}^*Q_{\mathcal{U}}$ is a vector bundle and $\mathcal{T}or^1(I_{X\times \{g\}}, \mathcal{O}_{\Gamma})=0,$ it follows that $\mathcal{T}or^1(I_{X\times \{g\}}, \mathcal{O}_{\Gamma}\otimes  p_{\mathbb{G}(\mathcal{U})}^*Q_{\mathcal{U}} )=p_{\mathbb{G}(\mathcal{U})}^*Q_{\mathcal{U}}\otimes \mathcal{T}or^1(I_{X\times \{g\}}, \mathcal{O}_{\Gamma})=0.$ Therefore,  restricting the exact sequence
$(\ref{sucuniversal})$  to $X\times \{((x,E),W)\}$,
 we get  the exact sequence
\begin{eqnarray}
0 \longrightarrow E^{W} \longrightarrow E \longrightarrow \mathcal{O}_x \otimes W \longrightarrow 0
\end{eqnarray}
over $X$. Moreover, if we restrict
(\ref{sucuniversal}) to $X\times \mathbb{G}(z)$ we obtain (\ref{famE}).

Hence by similar arguments to Proposition \ref{Family} we have that  $\mathcal{F}$ is a family of stable  torsion-free sheaves of  rank $n$
 of type $(c_1, c_2+m)$ which determines a morphism
 \begin{eqnarray*}
 \Phi:\mathbb{G}(\mathcal{U},m)&\rightarrow& \mathfrak{M}_{X,H}(n;c_1,c_2+m)\\
 ((x,E),W) &\mapsto & E^W.
\end{eqnarray*}

Note that $\text{Im}\,\Phi$ lies in $\mathfrak{M}_{X,H}(n;c_1,c_2+m)- M_{X,H}(n;c_1,c_2+m)$. In the following theorem we compute the dimension of  $\text{Im}\,\Phi$.

\begin{Theorem}\label{second}
Let $m,n$ natural integers with $1\leq m < n $. Then $\mathfrak{M}_{X,H}(n;c_1,c_2+m)-M_{X,H}(n;c_1,c_2+m)$ contains an irreducible projective variety
$Y$ of dimension $3+\dim\, M_{X,H}(n;c_1,c_2)$ such that the general element
$F \in Y$ fits into exact sequence
\begin{eqnarray*}
0 \to F \to E \to \mathcal{O}_{X,x}\otimes W \to 0,
\end{eqnarray*}
    where $E\in M_{X,H}(n;c_1,c_2)$, $W\in\mathbb{G}(E_x, m)$ and $x\in X$.
    In particular, if $n=2 $ then $\Phi$ is injective and $Y$  is a divisor.
\end{Theorem}

\begin{proof}
We will prove that image of $\Phi$ is an irreducible variety
of dimension $3 + \dim\, M_{X,H}(n;c_1,c_2)$.
For this, it will thus be sufficient to  compute the dimension
of the fibers of $\Phi$.
Let $F \in \text{Im} \,\Phi$, then there exists $((x,E),W)\in \mathbb{G}(\mathcal{U},m)$
such that $F$ fits into the following exact sequence
\begin{equation}\label{Imagen}
0 \to F \to E \to \mathcal{O}_{X,x}\otimes W \to 0,
\end{equation}
where $E$ is a vector bundle and $W\in \mathbb{G}(E_x,m)$.
We claim $\dim\, \text{Ext}^1(\mathcal{O}_{X,x}\otimes W,F)=m^2$.

From the exact sequence (\ref{Imagen}) we get the long exact sequence
\[\begin{aligned}
0  & \to \text{Hom}(\mathcal{O}_{X,x},F) \to
\text{Hom}(\mathcal{O}_{X,x},E)\to \text{Hom}(\mathcal{O}_{X,x},\mathcal{O}_{X,x}\otimes W) \to \\ & \text{Ext}^1(\mathcal{O}_{X,x},F) \to \text{Ext}^1(\mathcal{O}_{X,x},E) \to \text{Ext}^1(\mathcal{O}_{X,x},\mathcal{O}_{X,x}\otimes W) \to \ldots
\end{aligned}\]
Since $\text{Hom}(\mathcal{O}_{X,x},E)=0$ and
by Lemma \ref{Ext} $\text{Ext}^1(\mathcal{O}_{X,x},E)=0$,
it follows that
\[
\dim\, \text{Ext}^1(\mathcal{O}_{X,x},F) = \dim\,\text{Hom}(\mathcal{O}_{X,x},\mathcal{O}_{X,x}\otimes W)=m.
\]
Thus $\dim\,\text{Ext}^1(\mathcal{O}_{X,x}\otimes W,F)=m^2$.

We now proceed to compute the dimension of $Im \, \Phi$.
Let $p_i$ be denote the canonical projection of  $X
\times \mathbb{G}(E_x, m)$ for $i=1,2$ and consider  the sheaf
$\mathcal{H}om(p_{1}^*\mathcal{O}_{x}\otimes
p_2^{*}\mathcal{Q}_{E_{x}},
p_{1}^*F )$.
Taking higher direct image we obtain on $\mathbb{G}(E_x, m)$ the sheaf:
\begin{eqnarray*}
\Lambda := R^1_{p_{{2}_*}}\mathcal{H}om(p_{1}^*\mathcal{O}_{x}\otimes
p_2^{*}\mathcal{Q}_{E_{x}},
p_{1}^*F).
\end{eqnarray*}
This $\Lambda$ is
locally free over $\mathbb{G}(E_x, m)$ because
\begin{eqnarray*}
H^0(\mathcal{H}om(\mathcal{O}_{X,x} \otimes
W,F)) \cong
\text{Hom}(\mathcal{O}_{X,x} \otimes
W,F)=0,
\end{eqnarray*}
for any $W\in \mathbb{G}(E_x, m)$.
Hence the fiber of  $\Lambda$ at $W \in \mathbb{G}(E_x, m)$ is
$\text{Ext}^1(\mathcal{O}_{X,x} \otimes
W,F)$.

Let $\pi :\mathbb{P}\Lambda \to \mathbb{G}(E_x, m)$ denote  the
projectivization of the sheaf  $\Lambda$.
By \cite[Lemma 3.2]{Gottsche} there exists an exact sequence:
\begin{equation} \label{univextension2}
0 \to (id\times\pi)^*p_{1}^*F \otimes
\mathcal{O}_{X \times \mathbb{P}\Lambda}(1) \to
\mathcal{E} \to
(id\times\pi)^*(p_{1}^*\mathcal{O}_{X,x} \otimes
p_2^{*}\mathcal{Q}_{E_{x}}) \to 0
\end{equation}
on $X \times \mathbb{P}\Lambda$ such that, for each $p
\in \mathbb{P}\Lambda$, its restriction  to $X \times \{p\}$ is the extension
\begin{eqnarray*}
0 \longrightarrow F \longrightarrow \mathcal{E}_{|_p}  \longrightarrow
\mathcal{O}_{X,x}\otimes W \longrightarrow 0
\end{eqnarray*}

where $\mathcal{E}_{|_p}:=\mathcal{E}_{|_{X\times\{p\}}}$.

The set
\[U:= \{p \in  \mathbb{P}\Lambda ~|~\mathcal{E}_{|_p} \text{ is locally free and stable} \} \]
 is irreducible open set of dimension $ m(n-m)+m^2-1=mn-1.$ Therefore the dimension of the fiber of $\Phi$ is $mn-1-m^2=m(n-m)-1$ and then we have
 \[\begin{aligned}
 \dim \, Im\, \Phi &= m(n-m)+2+ \dim\, M_{X,H}(n; c_1,c_2)-m(n-m)+1 \\
 &= 3+\dim\, M_{X,H}(n;c_1,c_2).
 \end{aligned}\]

Note that for rank two case, the morphism $\phi$  is injective because the dimension of $\mathbb{P}\text{Ext}^1(\mathcal{O}_{X,x}\otimes W, F)=0$ and $\mathbb{P}\text{Ext}^1(\mathcal{O}_{X,x}\otimes W, F)$ is irreducible.
\end{proof}

By functorial construction we also have the following algebraic morphism
 \begin{eqnarray*}
 \Psi: X \times M_{X,H}(n;c_1, c_2)&\rightarrow &\text{Hilb}_{\ \mathfrak{M}_{X,H}(n;c_1,c_2+m)}\\
z=(x,E) &\mapsto &\mathbb{G}(z)
 \end{eqnarray*}
with $\mathbb{G}(z):= \phi_z(\mathbb{G}(E_x,m))$.  This construction is essentially the same as the one carried out in
\cite{Brambila-Mata, Narasimhan-Ramanan}.

The injectivity of the function $\Psi:X \times M_{X,H}(n;c_1, c_2)\rightarrow \text{Hilb}_{\ \mathfrak{M}_{X,H}(n;c_1,c_2+m)}$ is established in the next proposition.  The proof proceeds as
\cite[Proposition 3.2]{Brambila-Mata} and we use the following two lemmas.

\begin{Lemma}\label{cocientevectorbundle}
Let $X$ be an irreducible variety and let
\[0 \to F \to E \to G \to 0\]
be an exact sequence of sheaves over $X$.
If $E$ and $G$ are locally free sheaves, then
$F$ is locally free.
\end{Lemma}

\begin{proof}
Let $H$ be a sheaf on $X$.  We claim that for any locally free sheaf $E$ on $X$
$\mathcal{E}xt^{i}(E,H)=0$. By \cite[Proposition 6.8]{Hartshorne1} we have
\[\mathcal{E}xt^{i}(E,H)_x \cong \text{Ext}^i(E_x, H_x)\]
which is zero for any $x\in X$ because  \cite[Theorem 17]{Friedman}.  Consider the exact sequence \begin{equation} \label{seq}
    0 \to F \to E \to G \to 0
\end{equation}
where $E$ and $G$ are locally free sheaves.  Applying the functor $\mathcal{H}om(-,H)$ to the exact sequence (\ref{seq}) we get
\[\begin{aligned}
    0 \to & \mathcal{H}om(G,H) \to \mathcal{H}om(E,H) \to  \mathcal{H}om(F,H) \to \\
    & \mathcal{E}xt^1(G,H) \to \mathcal{E}xt^1(E,H) \to  \mathcal{E}xt^1(F,H) \to  \mathcal{E}xt^2(G,H) \to \cdots
\end{aligned}\]
Note that $\mathcal{E}xt^i(G,H) = \mathcal{E}xt^i(E,H) = 0$ for $i>0$.   Therefore, $\mathcal{E}xt^1(F,H) =0$ from which we conclude that
%by \cite[Exercise 6.5]{Hartshorne1},
$F$ is locally free as we desired.
\end{proof}

\begin{Lemma}[\cite{Huybrechts-Lehn}, Lemma 8.2.12]\label{Dpropo}
Let $F_1$ and $F_2$ be $\mu$-semistable sheaves on $X$. If $a$ is sufficiently
large integer and $C\in|aH|$ a general non-singular curve, then $F_1|_C$ and $F_2|_C$
are $S$-equivalent if and only if $F_1^{**}\cong F_2^{**}$
\end{Lemma}

\begin{Proposition}\label{primerPropo}
The morphism $\Psi:X \times M_{X,H}(n;c_1, c_2)\rightarrow \text{Hilb}_{\ \mathfrak{M}_{X,H}(n;c_1, c_2+m)}$
defined  as above is injective.
\end{Proposition}

\begin{proof}
Assume that for $i=1,2$ there exist $z_i=(x_i,E_i)\in X\times M_{X,H}(n;c_1,c_2)$
such that $\mathbb{G}(z_1)= \mathbb{G}(z_2)$, we want to prove that $E_1\cong E_2$
and $x_1=x_2$. We recall that for any $z_i=(x_i,E_i)$ there exists a family
 $\mathcal{F}_{z_i}$ of stable torsion-free sheaves parameterized by $\mathbb{G}(z_i),$
and $\mathcal{F}_{z_i}$ fits into the following exact sequence
  \begin{eqnarray}\label{famEi}
0 \longrightarrow \mathcal{F}_{z_i}
\longrightarrow p_1^*E_i \longrightarrow p_1^*\mathcal{O}_{x_i} \otimes p_2^*Q_{E_{x_i}} \longrightarrow 0
\end{eqnarray}
of sheaves over $X\times \mathbb{G}(z_i),$ where $p_j$ denotes the $j$-projection over $X\times \mathbb{G}(z_i)$. From universal properties of moduli space $\mathfrak{M}_{X,H}(n;c_1,c_2+m)$,
there exists an isomorphism $\beta:\mathbb{G}(z_1)\rightarrow \mathbb{G}({z_2})$
  that induces the following commutative diagrams
  \[\begin{diagram}
      \node{\mathbb{G}(z_1)}  \arrow[2]{e,l}{\beta} \arrow{se,l}{\phi_{z_1}} \node{} \node{\mathbb{G}(z_2)} \arrow{sw,l}{\phi_{z_2}}\\
      \node{} \node{\mathfrak{M}_{X,H}(n;c_1,c_2+m)}
  \end{diagram}\]
  and
  \[\begin{diagram}
  \node{X\times \mathbb{G}(z_1)} \arrow[2]{e,l}{id_{X} \times \beta} \arrow{se,l}{p_1} \node{} \node{X\times \mathbb{G}(z_2)}\arrow{sw,l}{p'_1} \\
  \node{} \node{X}
  \end{diagram}\]

  i.e. $\phi_{z_1}=\phi_{z_2}\circ\beta$ and $p_1 = p'_{1} \circ \, (id_{X} \times \beta)$.
  By the universal property of $\mathfrak{M}_{X,H}(n;c_1,c_2+m)$, we have
\[
\mathcal{F}_{z_1}\cong (id_{X}\times\beta)^*\mathcal{F}_{z_2}\otimes p_2^*(L)
\]
for some line bundle $L$ on $\mathbb{G}({z_1)}$. The following properties are satisfied:
\begin{enumerate}
 \item  $L$ is trivial.
 \item  $ R^1 {p_1}_*(\mathcal{F}_{z_1})= R^1 {p'_1}_*(\mathcal{F}_{z_2})= 0$.
 \item  ${p_1}_{*}\mathcal{F}_{z_1}= {p'_1}_{*} \mathcal{F}_{z_2}$.
\end{enumerate}
    First we proved that $\mathcal{F}_{z_i}|_{\{y\}\times \mathbb{G}(z_i)}\cong E_y\otimes \mathcal{O}_{\mathbb{G}(z_i)}$ is trivial for
 any $y\neq x_i$.
 Restricting the exact sequence (\ref{famEi}), we obtain
  \begin{eqnarray*}
 0 \rightarrow \mathcal{T}or^1(\mathcal{O}_{\mathbb{G}},p_1^*\mathcal{O}_{x_i} \otimes p_2^*Q_{E_{x_i}}) \rightarrow \mathcal{F}_{z_i}|_{y\times \mathbb{G}(z_i)}\rightarrow p_1^{*}(E_i)|_{y\times \mathbb{G}(z_i)}\rightarrow 0.
\end{eqnarray*}
Note that $p_1^{*}(E_i)|_{y\times \mathbb{G}(z_i)}\cong E_y \otimes \mathcal{O}_{\mathbb{G}(z_i)}$ and $\mathcal{F}_{z_i}|_{y\times \mathbb{G}(z_i)}$ are vector bundle of the same rank, then by Lemma \ref{cocientevectorbundle} we have $\mathcal{T}or^1(\mathcal{O}_{\mathbb{G}},p_1^*\mathcal{O}_{x_i} \otimes p_2^*Q_{E_{x_i}})=0$ and $\mathcal{F}_{z_i}|_{y\times \mathbb{G}(z_i)}\cong  E_y \otimes \mathcal{O}_{\mathbb{G}(z_i)}$. On the other hand
\begin{eqnarray*}
(id_X \times \beta)^{*}(\mathcal{F}_{z_2})|_{y\times \mathbb{G}(z_1)}=\beta^{*}(\mathcal{F}_{z_2}|_{y\times \mathbb{G}(z_2)})=\beta^{*}(E_y\otimes \mathcal{O}_{G(z_2)})=E_y\otimes \mathcal{O}_{G(z_1)}.
\end{eqnarray*}
Therefore
\begin{eqnarray*}
E_y\otimes \mathcal{O}_{G(z_1)}=\mathcal{F}_{z_1}|_{y \times \mathbb{G}(z_1)}\cong ((id_{X}\times\beta)^*\mathcal{F}_{z_2}\otimes p_2^*(L))|_{y\times \mathbb{G}(z_1)}= E_y\otimes \mathcal{O}_{G(z_1)}\otimes L.
\end{eqnarray*}
 Thus $L$ is trivial \cite[Page 12] {newstead} and  this prove (1). Moreover
 \begin{eqnarray*}
 \mathcal{F}_{z_1}|_{x_1 \times \mathbb{G}(z_1)}\cong ((id_{X}\times\beta)^*\mathcal{F}_{z_2})|_{x_1\times \mathbb{G}(z_1)}= \beta^{*}(E_{x_1}\otimes \mathcal{O}_{G(z_2)})=E_{x_1}\otimes \mathcal{O}_{\mathbb{G}(z_1)}.
 \end{eqnarray*}

 And for any $y \in X$ we have
\begin{eqnarray*}
R^1 {p_1}_{*}(\mathcal{F}_{z_1})_y= H^1(\mathcal{F}_{z_1}|_{y\times \mathbb{G}(z_1)})=H^1(E_y \otimes \mathcal{O}_{\mathbb{G}(z_1)})=0.
\end{eqnarray*}

Similarly, we can prove that $\mathcal{F}_{z_2}|_{x_2 \times \mathbb{G}(z_2)}\cong E_{x_2}\otimes \mathcal{O}_{\mathbb{G}(z_2)}$ and

$R^1 {p_1}^{'}_*(\mathcal{F}_{z_2})=0$ and this prove (2). Since $p_1=p_1^{'}\circ (id \times \beta)$ and  $(id_X \times \beta)$ is an isomorphism, we get
\begin{eqnarray*}
p_{1*}(\mathcal{F}_{z_1})=p_{1*}(id \times \beta)^{*}(\mathcal{F}_{z_2})&=&(p'_1 \circ (id \times \beta))_{*} (id\times \beta)^{*} \mathcal{F}_{z_2})\\
&=& p'_{1*}((id \times \beta)_{*}(id \times \beta)^{*}(\mathcal{F}_{z_2}))\\
&=& p'_{1_*}(\mathcal{F}_{z_2}),
\end{eqnarray*}
and this proves (3).  We now proceed to show that $E_1\cong E_2$ and $x_1=x_2$. The proof will be
divided into three steps:

\textbf{Step 1:} We will show that $E_1\otimes I_{x_1}\cong E_2\otimes I_{x_2}$.

Taking the direct image of (\ref{famEi}) by $p_1$ we obtain the following exact sequence:
\begin{eqnarray*}
0\rightarrow {p_1}_*(\mathcal{F}_{z_1})\rightarrow {p_1}_*(p^*_1 E_1)\rightarrow {p_1}_*(p_1^*\mathcal{O}_{x_1}\otimes p_2^*Q_{E_{1,x_1}})\rightarrow 0
\end{eqnarray*}
because  $ R^1 {p_1}_*(\mathcal{F}_{z_1})=0$. And we can complete the diagram
\begin{equation*}
 \xymatrix{0  \ar[r]  & E_1\otimes I_{x_1} \ar[r]^{} \ar@{^{}->}[d]^{}      &  E_1 \ar[r]^{} \ar@{^{}->}[d]^{{}{}} & E_1 \otimes \mathcal{O}_{x_1} \ar[r] \ar@{->}[d]^{}  & 0\\
				0\ar[r]^{}	& {p_1}_*(\mathcal{F}_{z_1}) \ar[r]_{} & {p_1}_*(p^*_1 E_1) \ar[r]_{}   & {p_1}_*(p_1^*\mathcal{O}_{x_1}\otimes p_2^*Q_{E_{1,x_1}}) \ar[r]^{} & 0    .\\  }
\end{equation*}

 Since  ${p_1}_*p_1^*(E_1)\cong E_1$ and ${p_1}_*(p_1^*\mathcal{O}_{x_1}\otimes p_2^*Q_{E_{1,x_1}})\cong E_{1}\otimes \mathcal{O}_{x_1}$ by projection formula, it follows that ${p_1}_*\mathcal{F}_{z_1}\cong E_1\otimes I_{x_1}$. We can now proceed analogously to obtain ${p'_1}_{*}\mathcal{F}_{z_2}\cong E_2\otimes I_{x_2}$. Therefore
 \begin{eqnarray*}
 E_1\otimes I_{x_1}\cong {p_1}_*\mathcal{F}_{z_1} \cong {p'_1}_{*}\mathcal{F}_{z_2}\cong E_2\otimes I_{x_2}.
 \end{eqnarray*}

\textbf{Step 2:} We will show that  $E_1 \cong E_2$; %We will need the following results.

Note that the general curve on $X$ does not goes through the points $x_1$ and $x_2$,  hence $E_1|_C\cong (E_1\otimes I_{x_1})|_{C}\cong (E_2\otimes I_{x_1})|_{C}\cong E_2|_C$
for the general curve $C\in |aH|$. From Lemma \ref{Dpropo}, we conclude that $E_1\cong E_2$
which is the desired conclusion.

\textbf{Step 3:} We show will that $x_1=x_2$;

Notice that by step 1 there exists an isomorphism $\lambda:E_1\otimes I_{x_1} \to E_2\otimes I_{x_2}$.  On the other hand step 2 provided us an isomorphism $\phi:E_1\rightarrow E_2$.  Considering
 the exact sequence
\[\begin{diagram}
  \node{0} \arrow{e}  \node{E_i \otimes I_{x_i}} \arrow{e,l}{f_i}  \node{E_i} \arrow{e,l}{\alpha_i} \node{E_i\otimes \mathcal{O}_{x_i}}  \arrow{e} \node{0}
\end{diagram}\]
    for $i=1,2$.  Moreover $\phi\circ f_{1},$ $f_{2}\circ \lambda\in \text{Hom}(E_1\otimes I_{x_1},E_2)$, and hence by Lemma $ \ref{INY}$,   $\phi\circ f_{1}=t(f_{2}\circ \lambda)$ for some $t\in \mathbb{C}^{*}$. Without loss of generality we suppose that $t=1$ therefore we have the following commutative diagram
\begin{equation*}
 \xymatrix{0  \ar[r]  & E_1\otimes I_{x_1} \ar[r]^{f_1} \ar@{^{}->}[d]^{\lambda}      &  E_1 \ar[r]^{\alpha_1} \ar@{^{}->}[d]^{{\phi}{}} & E_1\otimes \mathcal{O}_{{x_1}} \ar[r] \ar@{->}[d]^{\alpha}  & 0\\
				0\ar[r]^{}	& E_2\otimes I_{x_2} \ar[r]_{f_2} & E_2 \ar[r]_{\alpha_2}   & E_2\otimes \mathcal{O}_{x_2} \ar[r]^{} & 0    ,\\  }
\end{equation*}
where $\alpha$ is an isomorphism of skyscraper sheaves supported at $x_1$ and $x_2$, respectively. Hence $x_1=x_2$. Therefore $\Psi$ is injective which establishes the proposition.
\end{proof}

We can now state our main result. The  theorem computes a bound of the dimension of an irreducible subvariety of the Hilbert scheme $\text{Hilb}_{\mathfrak{M}_{X,H}(n;c_1,c_2+m)}$.

\begin{Theorem} \label{principal}
The Hilbert scheme $\text{Hilb}_{\mathfrak{M}_{X,H}(n;c_1,c_2+m)}$ of the moduli space of stable vector bundles has an irreducible component of dimension at least $2+\dim\, \mathfrak{M}_{X,H}(n;c_1,c_2+m)$.
\end{Theorem}

\begin{proof}
The proof follows from Proposition \ref{primerPropo}.
\end{proof}

\section{Application to the moduli space of sheaves on the projective plane}

Let us denote by $\mathfrak{M}_{\mathbb{P}^2}(2;c_1,c_2)$ the moduli space of rank $2$ stable sheaves on the projective plane $\mathbb{P}^2$ with respect to the ample line bundle $\mathcal{O}_{\mathbb{P}^2}(1)$.
By Proposition \ref{encajeG}, the image $\phi_{z}(\mathbb{P}(z))$ defines a cycle in the Hilbert scheme of $\mathfrak{M}_{\mathbb{P}^2}(2;c_1,c_2)$

In this section we will describe the component of the Hilbert scheme which contains the cycles
$\phi_z(\mathbb{P}(E_x))$.  Our computations use some results and
techniques of \cite{hulek} and \cite{Stromme}.

\begin{Definition}
\begin{em}
    Let $E$ be a normalized rank $2$ sheaf on $\mathbb{P}^2$. A line $L$ (resp. a conic $C$) $\subset \mathbb{P}^2$ is  jumping line (resp. jumping
conic) if $h^1 (L,E(-c_1-1) \vert_L)\neq 0$ (resp. $h^1 (C,E\vert_C) \neq 0)$.
\end{em}
\end{Definition}

The following theorem was proved in \cite{Stromme}
\begin{Theorem}\label{Stromme}
Assume that $c_1=-1$ (resp. $c_1=0$) and that
$c_2= n \geq 2$ (resp. $c_2= n\geq 3$    is odd). Then
\begin{itemize}
    \item[(i)] $Pic(\mathfrak{M}_{\mathbb{P}^2}(2;c_1,c_2))$ is freely generated   by two generators denoted by $\epsilon$ and $\delta$ (resp. $\varphi$ and $\psi$).
    \item[(ii)] An integral linear combination $a\epsilon + b\delta$ (resp.  $a \varphi +b\psi$) is ample if and only
if $a>0$ and $b>0$.
    \item[(iii)]Consider the following sets in
   $ \mathfrak{M}_{\mathbb{P}^2}(2;c_1,c_2)$:
   \begin{eqnarray*}
   D_1 &=&\{\mbox{sheaves with a given jumping conic (resp- line)}\}.\\
   D_2&=&\{\mbox{sheaves with a given jumping line (resp. conic)  passing through 1 (resp. 3) given points}\}.
   \end{eqnarray*}
  Then $D_1$ is the support of a reduced divisor in the linear system
$ \vert  \epsilon \vert$ (resp. $ \vert \varphi \vert$)  and $D_2$ is the support of a reduced divisor in the
linear system $\vert\delta \vert$ (resp. $ \vert \frac{1}{2}(n-1)\psi\vert $).
\end{itemize}
\end{Theorem}

Following the construction given in Section 3,
if  $z=(x,E) \in \mathbb{P}^2 \times M_{\mathbb{P}^2}(2;c_1,c_2-1)$ then, Proposition \ref{Family},  we have a family $\mathcal{F}_z$ of $H$-stable torsion-free sheaves rank two on $\mathbb{P}^2$ parameterized by $\mathbb{P}(E_x)$
or $\mathbb{P}(z)$ for short. Such family fits in the following exact sequence
\begin{eqnarray} \label{exactp2}
0 \longrightarrow \mathcal{F}_z
\longrightarrow p_1^*E \longrightarrow p_1^*\mathcal{O}_x \otimes p_2^*Q_{E_x} \longrightarrow 0,
\end{eqnarray}
defined on $\mathbb{P}^2\times \mathbb{P}(z).$
The classification map of $\mathcal{F}_z$ is the morphism  \begin{equation} \label{morphismp2}
    \phi_z: \mathbb{P}(z) \to \mathfrak{M}_{\mathbb{P}^2}(2;c_1,c_2)
\end{equation}
defined as $\phi_z(W)=E^{W}$.

We now use the exact sequence (\ref{exactp2}) and the morphism  (\ref{morphismp2}) to determine the irreducible component of the Hilbert scheme $\text{Hilb}_{\mathfrak{M}_{\mathbb{P}^2}(2;c_1,c_2)}$ of the moduli space $\mathfrak{M}_{\mathbb{P}^2}(2;c_1,c_2)$, $c_1=0$ or $-1$ which contains the cycles
$\phi_z(\mathbb{P}(z))$.  This component is denoted  by $\mathcal{H}\mathcal{G}$.

For the proof of the theorem, we first  establish the result for two particular cases: $c_1=-1$ and $c_1=0$.

\begin{Theorem}\label{princlemma}
Under the notation of Theorem \ref{Stromme}
\begin{enumerate}
    \item Assume that $c_1=-1$  and let
$c_2 \geq 2$. Let  $\mathfrak{L}:= a\epsilon + b\delta$ be an ample line bundle in $Pic(\mathfrak{M}_{\mathbb{P}^2}(2;c_1,c_2))$. Then,  $\mathcal{H}\mathcal{G}$  is the component of the Hilbert scheme $\text{Hilb}^P_{\mathfrak{M}_{\mathbb{P}^2}(2;c_1,c_2)}$  where $P$ is the Hilbert polynomial defined as;
\[P(m) = \chi(\mathbb{P}(z), \phi^{*}_z(\mathfrak{L}) ) = \chi(\mathbb{P}(z),  \mathcal{O}_{\mathbb{P}(z)}(mb)).\]

\item Assume that  $c_1=0$ and let $c_2 \geq 3$  odd number. Let   $\mathfrak{L}:= a\varphi+b\psi$ be an ample line bundle in $Pic(\mathfrak{M}_{\mathbb{P}^2}(2;c_1,c_2))$. Then,  $\mathcal{H}\mathcal{G}$  is the component of the Hilbert scheme $\text{Hilb}^P_{\mathfrak{M}_{\mathbb{P}^2}(2;c_1,c_2)}$  where $P$ is the Hilbert polynomial defined as;
\[P(m) = \chi(\mathbb{P}(z), \phi^{*}_z(\mathfrak{L}) ) = \chi(\mathbb{P}(z),  \mathcal{O}_{\mathbb{P}(z)}(m(c_2-1)b))).\]
\end{enumerate}
\end{Theorem}

\begin{proof}
\begin{enumerate}
    \item Let $z=(x,E) \in \mathbb{P}^2 \times M_{\mathbb{P}^2}(2;c_1,r)$, $c_1=-1$ and $r \geq 1.$ Consider the family $\mathcal{F}_z$ of stable sheaves of rank two given by  the exact sequence (\ref{exactp2}). Then,

 $\mathcal{F}_{z_t} :=(\mathcal{F}_{z})\vert_{\mathbb{P}^2 \times \{t\}}$ is stable for any $t\in \mathbb{P}(z)$ and by  Proposition \ref{properties}  its Chern classes are $c_1(\mathcal{F}_{z_t})=-1$
 and $c_2:= c_2(\mathcal{F}_{z_t})=r+1 \geq 2$.  Therefore we have  the morphism
\[\phi_z: \mathbb{P}(E_x) \longrightarrow \mathfrak{M}_{\mathbb{P}^2}(2;c_1,c_2), \,\,\,\, t \mapsto \mathcal{F}_z \vert_t\]
and set $\tau = p_1^*(\mathcal{O}_{\mathbb{P}^2}(1))$.

Now we will compute $\phi_z^* \epsilon$  and $\phi_z^* \delta$.

Let $l \geq 0 $, from the exact sequence (\ref{exactp2}) we have
\[\begin{aligned}
0  \to & p_{2_*}\mathcal{F}(-l\tau) \to p_{2_*}p_1^*E(-l\tau) \to p_{2_*}p_1^*\mathcal{O}_x(-l\tau) \otimes p_2^*Q_{E_x} \to \\ & R^1p_{2_*}\mathcal{F}(-l\tau) \to
 R^1p_{2_*}p_1^*E(-l\tau) \to R^1p_{2_*}(p_1^*\mathcal{O}_x(-l\tau) \otimes p_2^*Q_{E_x}) \to 0.
\end{aligned}\]
Using the projection formula, we get
\begin{eqnarray*}
R^ip_{2_*}p_1^*E(-l\tau) = \mathcal{O}_{\mathbb{P}(E_x)}\otimes H^i(\mathbb{P}^2,E(-l)).
\end{eqnarray*}
Since  $E(-l)$ is  a stable vector bundle on $\mathbb{P}^2$ with $c_1\leq 0,$ it follows that $p_{2_*}p_1^*E(-l\tau) =0$ and $R^ip_{2_*}p_1^*E(-l\tau)$ is a trivial bundle. Moreover, by similar arguments we have
\begin{eqnarray*}
R^ip_{2_*}(p_1^*\mathcal{O}_x(-l\tau) \otimes p_2^*Q_{E_x}) \cong Q_{E_x} \otimes p_{2_*}p_1^*\mathcal{O}_x(-l\tau) \cong Q_{E_x}\otimes H^i(\mathbb{P}^2,\mathcal{O}_{\mathbb{P}^2}(-l)_x).
\end{eqnarray*}
Hence $R^1p_{2_*}p_1^*\mathcal{O}_x(-l\tau) \otimes p_2^*Q_{E_x}=0$ and $p_{2_*}p_1^*\mathcal{O}_x(-l\tau) \otimes p_2^*Q_{E_x}=Q_{E_x}$. Therefore,
we have the exact sequence
\begin{eqnarray*}
0 \to Q_{E_x} \to R^1p_{2_*}\mathcal{F}(-l\tau) \to
 R^1p_{2_*}p_1^*E(-l\tau) \to 0
 \end{eqnarray*}
 where we conclude that
 $
 c_1(R^1p_{2_*}\mathcal{F}(-l\tau)) = 1
$
 for any $l\geq 0$.

 According to \cite[Lemmas 3.3 and  3.4]{hulek}, it follows that
 \[
 \phi_z^*(\epsilon) = c_1(R^1p_{2_*}\mathcal{F})- c_1(R^1p_{2_*}\mathcal{F}(-2\tau)=0 \]
 and
 \[ \phi_z^*(\delta) = (r+1)c_1(R^1p_{2_*}\mathcal{F})-rc_1(R^1p_{2_*}\mathcal{F}(-\tau))= 1. \]

 Hence, we conclude that \[P(m) = \chi(\mathbb{P}(z), \phi^{*}_z(a\epsilon+b\delta) ) = \chi(\mathbb{P}(z),  \mathcal{O}_{\mathbb{P}(E_x)}(mb))\]
 as we desired.

\item For the case $c_1=0$ and $c_2 \geq 3$ odd.  Consider $z=(x,E) \in \mathbb{P}^2 \times M_{\mathbb{P}^2}(2;c_1,r)$, $c_1= 0$ and $r \geq 2$ even.  From the exact sequence (\ref{exactp2}) we get $\mathcal{F}_{z_t} :=\mathcal{F}_{z_{\vert_{\mathbb{P}^2 \times \{t\}}}}$ is stable for all $t \in \mathbb{P}(E_x)$ and $c_1(\mathcal{F}_{z_t})=0$, $c_2:= c_2(\mathcal{F}_{z_t})=r+1 \geq 3$ odd.  By \cite[Lemmas 3.3 and  3.4]{hulek} we have that
\[\phi_z^*(\varphi) = c_1(R^1p_{2_*}\mathcal{F}(-\tau))- c_1(R^1p_{2_*}\mathcal{F}(-2\tau))=0,
\]
 and
\[
  \phi_z^*(\psi) = \frac{1}{2}r\left((r+1)c_1(R^1p_{2_*}\mathcal{F})-(r-1)c_1(R^1p_{2_*}\mathcal{F}(-\tau))\right)=c_2-1.\]
which implies
\[P(m) = \chi(\mathbb{P}(z), \phi^{*}_z(a\varphi+b\psi) ) = \chi(\mathbb{P}(z),  \mathcal{O}_{\mathbb{P}(E_x)}(m(c_2-1)b)))\]
and the proof is complete.
\end{enumerate}
\end{proof}


\begin{thebibliography}{999}


\bibitem {Barth}
  {\sc Barth W.P.; Hulek K.; Peters C.A.M ;  Van de Ven A.} ---
  {\it Compact complex surfaces}. Second edition. Ergebnisse der Mathematik und ihrer Grenzgebiete.
  Springer-Verlag, Berlin, 2004. xii+436 pp.


\vspace{.1cm}

\bibitem {Brambila-Mata}
  {\sc Brambila-Paz L. ; Mata-Guti\'errez O.} ---
  {\it On the Hilbert Scheme of the moduli space of vector bundles over an algebraic curve}.
  Manuscripta math. 142, 525–544 (2013).

  \vspace{.1cm}

\bibitem {CM}
  {\sc Costa L. ; Mir\'o-roig  R.M.} ---
  {\it Elementary transformations and the rationality
of the Moduli Spaces of Vector Bundles on
$\mathbb{P}^2.$
}  Manuscripta Math. 113 (2004), no. 1, 69–84.



\vspace{.1cm}

\bibitem {Coskun-Huizenga}
  {\sc Coskun I. ; Huizenga J.} ---
  {\it The ample cone of moduli spaces of sheaves on surfaces and the Brill-Noether Problem}.
  Manuscripta math. 142, 525–544 (2013).


\vspace{.1cm}

  \bibitem {Coskun-Huizenga1}
  {\sc Coskun I. ; Huizenga J.} ---
  {\it Brill-Noether theorems and globally generated vector bundles on Hirzebruch surfaces.}
  Nagoya Math. J. 238 (2020), 1–36.


\vspace{.1cm}

  \bibitem {Coskun-Huizenga2}
  {\sc Coskun I. ; Huizenga J.} ---
  {\it The birational geometry of the moduli spaces of sheaves on $\mathbb{P}^2$.}
  Proceedings of the G\"okova Geometry-Topology Conference 2014, 114–155, G\"okova Geometry/Topology Conference (GGT), G\"okova, 2015.

  \vspace{.1cm}

  \bibitem {Coskun-Huizenga3}
  {\sc Coskun I. ; Huizenga J.} ---
  {\it
The Moduli Spaces of Sheaves
on Surfaces, Pathologies
and Brill-Noether Problems.}  Geometry of moduli, 75–105, Abel Symp., 14, Springer, Cham, 2018.



\vspace{.1cm}

\bibitem{Eisenbud}
  {\sc  Eisenbud D.} ---
  {\it 3264 and all that a second course on algebraic geometry}. Cambridge University Press, 2016.


\vspace{.1cm}


\bibitem {Fantechi}
  {\sc Fantechi B. et. al.}---
  {\it Fundamental Algebraic Geometry: Grothendieck’s FGA Explained. }
  Mathematical Surveys and Monographs. 123 (2005).

\vspace{.1cm}


\bibitem {Friedman}
  {\sc Friedman R.} ---
  {\it Algebraic surfaces and holomorphic vector bundles}. Universitext. Springer-Verlag, New York, 1998. x+328 pp.


\vspace{.1cm}


\bibitem {Gottsche}
  {\sc  Gottsche L.} ---
  {\it Change of polarization and Hodge numbers of moduli spaces of torsion free sheaves on surfaces.}
 Mathematische Zeitschrift, 223, 247-260, (1996).


\vspace{.1cm}


 \bibitem {LePotier}
  {\sc  Le Potier J. } ---
  {\it Lectures on vector bundles}. Translated by A. Maciocia. Cambridge Studies in Advanced Mathematics, 54. Cambridge University Press, Cambridge, 1997. viii+251 pp.

\vspace{.1cm}



\vspace{.1cm}


\bibitem {Hartshorne1}
  {\sc \ Hartshorne R. } ---
   {\it Algebraic geometry}.
   Graduate Texts in Mathematics, No. 52. Springer-Verlag, New York-Heidelberg, 1977. xvi+496 pp.

\vspace{.1cm}


\bibitem{hulek}
  {\sc  Hirshowitz A.;  Hulek K.} ---
    {\it Complete Families of Stable Vector Bundles over $\mathbb{P}_2$.} Complex Analysis and Algebraic Geometry. Lecture Notes in Mathematics, Springer Verlag 1985.

\vspace{.1cm}


\bibitem {Huybrechts-Lehn}
  {\sc D.\ Huybrechts, M.; \ Lehn.} ---
   {\it The geometry of moduli spaces of sheaves.} Aspects of Mathematics, E31. Friedr. Vieweg  Sohn, Braunschweig, 1997. xiv+269 pp.

\vspace{.1cm}


\bibitem {Maruyama1}
 {\sc  Maruyama M.} ---
  {\it On a family of algebraic vector bundles.} Number theory, Algebraic Geometry and Commutative Algebra, in honor of Y. Akizuki Kinokuniya, Tokio (1973) 95-146.

\vspace{.1cm}

\bibitem {Maruyama2}
 {\sc  Maruyama M.} ---
  {\it Openness of a family of torsion free sheaves.} J. Math Kyoto Univ. 16-3 (1976) 627-637.

\vspace{.1cm}

\bibitem{mumfordcurvas}
 {\sc  Mumford D. } ---
 {\it Lectures on curves on an algebraic surface}. With a section by G. M. Bergman. Annals of Mathematics Studies, No. 59 Princeton University Press, Princeton, N.J. 1966 xi+200 pp.


\vspace{.1cm}

\bibitem{mumford}
  {\sc Mumford D.} ---
    {\it Lectures on curves on algebraic Surfaces.} Princeton University Press, 1985.


\vspace{.1cm}


\bibitem {Narasimhan-Ramanan}
  {\sc  Narasimhan M.S. ; Ramanan S.} ---
  {\it Geometry of Hecke Cycles I}.
C. P. Ramanujam-a tribute, Tata Inst. Fund. Res.
Stud. Math., Springer, Berlin-New York, 8 (1978), 291–345.
\vspace{.1cm}


\bibitem {NR}
  {\sc  Narasimhan M.S.; Ramanan S.} ---
  {\it Deformations of the moduli space of vector bundles over an algebraic curve}. Annals of Mathematics 101, no. 3 (1975): 391–417. https://doi.org/10.2307/1970933.

\vspace{.1cm}


\bibitem{newstead}
 {\sc Newstead P. E.} --- {\it Introduction to moduli problems and orbit spaces.} Tata Institute of Fundamental Research Lectures on Mathematics and Physics, 51. Tata Institute of Fundamental Research, Bombay; by the Narosa Publishing House, New Delhi, 1978. vi+183 pp.


\vspace{.1cm}


\bibitem {Okonek}
  {\sc  Okonek C.; Schneider M;  Spindler  H.} ---
  {\it Vector bundles on complex projective spaces}. Progress in Mathematics, 3. Birkhäuser, Boston, Mass., 1980. vii+389 pp.


\vspace{.1cm}


\bibitem{Stromme}
{\sc Strømme  S. A.} --- {\it Ample divisors on fine moduli spaces on the projective plane.} Math. Z. 187 (1984), no. 3, 405–423.

\vspace{.1cm}


\bibitem{Tyurin}
  {\sc   Tyurin A.} ---
  {\it Vector bundles}. Collected Works, Volume I.
  Edited by Fedor Bogomolov, Alexey Gorodentsev,
  Victor Pidstrigach, Miles Reid, and Nikolay Tyurin.
  Universitätsverlag Göttingen, 2008. vii+389 pp.

%\bibitem{Vakil}
 %       {\sc Vakil R.} ---
  %      {\it The Rising Sea: Foundations Of Algebraic Geometry.} Available at http://math.stanford.edu/ vakil/216blog/FOAGnov1817public.pdf.}



\end{thebibliography}
\end{document}